\documentclass[a4paper]{article}
\usepackage{amsmath}
\usepackage{amsthm}
\newtheorem{theorem}{Theorem}

\usepackage[colorlinks]{hyperref}
\usepackage{mathtools}
\mathtoolsset{showonlyrefs=true}

\newcommand{\dif}{\,\mathrm{d}}
\usepackage[a4paper, truedimen, hmargin=15mm, vmargin=20mm]{geometry}
\usepackage{hyperref}

\usepackage{newtx}

\title{A preconditioning technique of Gauss--Legendre quadrature for the logarithm of symmetric positive definite matrices}
\author{
  Fuminori Tatsuoka\footnote{
    Department of Applied Physics, Graduate School of Engineering, Nagoya University, Furo-cho, Chikusa-ku, Nagoya, 464-8603, Aichi, Japan.
    \texttt{\{f-tatsuoka,sogabe,kemmochi,zhang\}@na.nuap.nagoya-u.ac.jp}
  }
  \and Tomohiro Sogabe$^\ast$
  \and Tomoya Kemmochi$^\ast$
  \and Shao-Liang Zhang$^\ast$
}

\begin{document}

\maketitle

\begin{abstract}
  This note considers the computation of the logarithm of symmetric positive definite matrices using the Gauss--Legendre (GL) quadrature.
  The GL quadrature becomes slow when the condition number of the given matrix is large.
  In this note, we propose a technique dividing the matrix logarithm into two matrix logarithms, where the condition numbers of the divided logarithm arguments are smaller than that of the original matrix.
  Although the matrix logarithm needs to be computed twice, each computation can be performed more efficiently, and it potentially reduces the overall computational cost.
  It is shown that the proposed technique is effective when the condition number of the given matrix is approximately between $130$ and $3.0\times 10^5$.
\end{abstract}

\section{Introduction}\label{sect.introduction}
A logarithm of a square matrix $A \in {\mathbb{R}}^{n \times n}$ is defined as any matrix $X$ such that
\begin{align}\label{eq.def.logarithm}
  \exp(X) = A,
\end{align}
where $\exp(X) := I + X + X^{2}/2! + X^{3}/3! + \cdots$.
If all eigenvalues of $A$ lie in $\{ z \in {\mathbb{C}}:z \notin ( - \infty,0]\}$, there exists a unique solution of \eqref{eq.def.logarithm} whose eigenvalues are all in $\left\{ z \in {\mathbb{C}}:|\Im(z)| < \pi \right\}$.
Such a solution is called the principal logarithm and is denoted by $\log(A)$, see, e.g. \cite{higham2008functions} for more details.
The matrix logarithm arises in several situations of scientific computing, such as the von Neumann entropy in quantum information theory \cite{bengtsson2006geometry} and the log-determinant in machine learning \cite{han2015large}.

This note considers quadrature-based algorithms for computing the principal logarithm of symmetric positive definite (SPD) matrices.
As the name suggests, quadrature-based algorithms are methods that compute an integral representation of the matrix logarithm, such as
\begin{align} \label{eq.integral-representation}
  \log(A) = (A - I)\int_{0}^{1}\left\lbrack u(A - I) + I \right\rbrack^{- 1}du = (A-I) \int_{- 1}^{1}F(t;A) \dif t,
\end{align}
where $F(t; A) := [(1-t)I + (1+t)A]^{-1}$ and $u(t) = 2t-1$ \cite[Thm.~11.1]{higham2008functions}.
Compared to other computational methods for $\log(A)$, quadrature-based algorithms have the advantage of being well-suited for parallel computation and can be directly applied to the computation of the action of the matrix logarithm $\log(A)\boldsymbol{b}$ ($\boldsymbol{b} \in {\mathbb{R}}^{n}$).
See, e.g., \cite[Sect.~18]{trefethen2014exponentially} and \cite{tatsuoka2020algorithms} for more details.
For the computation of \eqref{eq.integral-representation}, the Gauss--Legendre (GL) quadrature and the double exponential formula are considered, see, e.g. \cite{tatsuoka2020algorithms,dieci1996computational}.

The motivation of this study is to improve the speed of the GL quadrature based on the equation
\begin{align} \label{eq.def.preconditioning}
  \log(A) = \log(AP) - \log(P)
\end{align}
with a suitable matrix $P \in {\mathbb{R}}^{n \times n}$.
If both $\log(AP)$ and $\log(P)$ can be computed more efficiently than $\log(A)$, the total computational cost for the right-hand side can be smaller than that of the left-hand side.
We refer to \eqref{eq.def.preconditioning} as ``preconditioning'' and to $P$ as the ``preconditioner,'' in accordance with the terminology used in the literature on numerical methods for linear systems.

Focusing on the SPD matrices, we propose a preconditioner of the form $P = (A + sI)^{-1} ~ (s \ge 0)$.
To the best of our knowledge, there are few studies on the preconditioning of quadrature-based algorithms for $\log(A)$.
In the literature \cite[Sect.~3]{fasi2018computing}, it is reported that multiplying $A$ by an appropriate scalar constant $c$ improves the convergence, which can be regarded as preconditioning with $P = cI$.
However, even after such scaling, the GL quadrature converges slowly when the condition number $\kappa(A)$ is large.
Here, we try reducing the condition number of both $AP$ and $P$ with the matrix-type preconditioner.
We study the selection of $s$, and we demonstrate that the GL quadrature with the proposed preconditioner is faster than existing algorithms when $\kappa(A)$ is approximately between $130$ and $3.3\times 10^5$.

The organization of this note is as follows.
In Section 2, we review the convergence of the GL quadrature for $\log(A)$.
In Section 3, we propose the preconditioner and show its convergence rate.
In Section 4, we present numerical results, and we give conclusions in Section 5.

\section{Convergence of the GL quadrature for the logarithm of SPD matrices}
In this section, we review the convergence rate of the GL quadrature for the logarithm of SPD matrices.
Let $A$ be an SPD matrix, $\Lambda$ be the spectrum of $A$, and $\lambda_{\max}$ and $\lambda_{\min}$ be the maximum and the minimum eigenvalue of $A$ respectively.
Then the error of a given quadrature formula for $\log(A)$ can be analyzed via that for the scalar logarithm:
\begin{align} \label{eq.scalar-logarithm}
  \left\| {\log(A) - \sum_{k = 1}^{m}w_{k}F\left( t_{k};A \right)} \right\|_{2} = \max_{\lambda \in \Lambda}\ \left| {\log(\lambda) - \sum_{k = 1}^{m}w_{k}F\left( t_{k};\lambda \right)} \right|,
\end{align}
where $m$ is the number of abscissas, and $\left\{ t_{k} \right\}_{k = 1}^{m}$ and $\left\{ w_{k} \right\}_{k = 1}^{m}$ are the abscissas and the weights of the given $m$-point quadrature formula, respectively.

The error of the GL quadrature for \eqref{eq.integral-representation} can be estimated by using the result in \cite[Sect.~3]{fasi2018computing}.
In the study \cite{fasi2018computing}, the integral $\int_{-1}^1 (1-t)^{-\alpha}(1+t)^{\alpha-1} [(1+t)A + (1-t)I]^{-1} \dif t ~ (\alpha \in (0, 1))$ is considered, and the Gauss--Jacobi quadrature with the weight $(1-t)^{-\alpha}(1+t)^{\alpha-1}$ is analyzed.
Because the integrand without the weight is the same as the integrand in \eqref{eq.integral-representation}, we can apply the estimate in \cite[Sect.~3]{fasi2018computing} to \eqref{eq.integral-representation}.
Using the results, the error of the $m$-point GL quadrature for $\log(\lambda)$ is estimated as
\begin{align}
  \left| {\log(\lambda) - \sum_{k = 1}^{m}w_{k}F\left( t_{k};\lambda \right)} \right|
  \le K \exp\Bigl(-\rho(\lambda) m\Bigr),
  \quad \rho(\lambda) = \log \left(\frac{(1 + \sqrt{\lambda})^2}{(1 - \sqrt{\lambda})^2}\right),
\end{align}
where $K$ is a scalar constant independent of $m$.
The convergence of the GL quadrature for $\log(\lambda)$ can be slow when $|\log(\lambda)|$ is large.
This is because the function $g(\lambda) = (1 + \sqrt{\lambda})^2 / (1 - \sqrt{\lambda})^2$ monotonically decreases for $\lambda > 1$ and satisfies $g(\lambda) = g(1/\lambda)$.

Hence, the error of the GL quadrature for $\log(A)$ depends on the extreme eigenvalues of $A$.
When the extreme eigenvalues are far from 1, i.e., $\lambda_{\max} \gg 1$ or $\lambda_{\min} \ll 1$, the convergence becomes slow.
Thus, it is better to scale $A$ to $\widetilde{A} = cA$ with $c = 1/\sqrt{\lambda_{\max} \lambda_{\min}}$ so that $g(c \lambda_{\max}) = g(c\lambda_{\min})$.
We note that $A$ can be scaled without assumptions because $\log(cA) = \log(A) + \log(c)I$ for $c>0$.
Finally, the error of the GL quadrature for $\log(\widetilde{A})$ is
\begin{align}\label{eq.err-gl-logm}
  \left\| {\log(\widetilde{A}) - \sum_{k = 1}^{m}w_{k}F\left( t_{k};\widetilde{A} \right)} \right\|_2 \le K'\exp\Bigl(-\rho\bigl(\kappa(A)\bigr) m\Bigr),
  \quad
  \rho(\kappa) = 2\log\left( \frac{\kappa^{1/4} + 1}{\kappa^{1/4} - 1} \right).
\end{align}
with the scalar constant $K'$ independent of $m$.
The coefficient $\rho(\kappa(A))$ in \eqref{eq.err-gl-logm} decreases as $\kappa(A)$ increases, and therefore the convergence becomes slow as $\kappa(A)$ increases.

\section{Preconditioning}\label{sect.preconditioning}
Our preconditioning technique is applied after the scaling.
Hence, the preconditioner is of the form $\widetilde{P}_s := (\widetilde{A} + sI)^{-1}$, where $\widetilde{A} = cA$, $c = 1 / \sqrt{\lambda_{\max} \lambda_{\min}}$, and $s \ge 0$.
With the preconditioning, $\log(A)$ is computed by
\begin{align}\label{eq:actual-preconditioning}
  \log(A) &= \log(\widetilde{A}) - \log(c)I\\
  & = \log(\widetilde{A}\widetilde{P}_s) - \log(\widetilde{P}_s) - \log(c)I\\
  & = \log(c'\widetilde{A}\widetilde{P}_s) - \log(c''\widetilde{P}_s) - \left[\log(c') - \log(c'') +  \log(c)\right]I\\
  & = \log(c'\widetilde{A}\widetilde{P}_s) - \log(c''\widetilde{P}_s) - \log(c)I, 
\end{align}
where $c$ and $c''$ are respectively the reciprocals of the geometric mean of the extreme eigenvalues of $\widetilde{A}\widetilde{P}_s$ and $\widetilde{P}_s$, and some calculation leads to $c' = c'' = \sqrt{(c\lambda_{\max}+s)(c\lambda_{\min} + s)}$.
In this section, we first check that $\widetilde{P}_s$ can be used as a preconditioner.
Then, we discuss the choice of $s$ that reduces the condition numbers of both $\widetilde{A}\widetilde{P}_s$ and $\widetilde{P}_s$, and we consider the convergence rate of preconditioned algorithms.

First, let us recall the following theorem about the sufficient condition to use a matrix as a preconditioner.
\begin{theorem}[see {\cite[Thm.~11.3]{higham2008functions}}] \label{thm.sufficient_condition}
  Suppose that both $A,P \in {\mathbb{C}}^{n \times n}$ have no eigenvalues on $\{ z \in {\mathbb{C}}:z \notin ( - \infty,0\rbrack\}$ and that $AP = PA$.
  For each eigenvalue $\lambda$ of $A$ there is an eigenvalue $\mu$ of $P$ such that $\lambda + \mu$ is an eigenvalue of $A + P$, and the eigenvalue $\mu$ is called the eigenvalue corresponding to $\lambda$.
  If $|\arg\lambda + \arg\mu| < \pi$ holds for every eigenvalues $\lambda$ of $A$ and its corresponding eigenvalue $\mu$ of $P$, then $\log(AP) = \log(A) + \log(P)$.
\end{theorem}
The matrices $\widetilde{A}$ and $\widetilde{P}_s$ satisfy the assumptions of Theorem \ref{thm.sufficient_condition} because $\widetilde{A}$ and $\widetilde{P}_s$ are commutative and all eigenvalues of $\widetilde{A}$ and $\widetilde{P}_s$ are positive.
Therefore, $\widetilde{P}_s$ can be used as a preconditioner because \eqref{eq.def.preconditioning} is satisfied for $A = \widetilde{A}$ and $P = \widetilde{P}_s$.

Now, we propose
\begin{align}
  \widetilde{P}_{1} = (\widetilde{A} + I)^{-1}
\end{align}
as the preconditioner because $s = 1$ minimizes the condition numbers of $\widetilde{A}\widetilde{P}_s$ and $\widetilde{P}_s$ according to the following theorem:
\begin{theorem}
  Let $A \ne I$ be an SPD matrix, and $\widetilde{A} = cA$ where $c = 1 / \sqrt{\lambda_{\max} \lambda_{\min}}$.
  Then, the following two equations hold:
  \begin{align}
    1 = \operatorname{argmin}\limits_{s \ge 0}\ \max\left\{ \kappa(\widetilde{A}\widetilde{P}_s),\kappa(\widetilde{P}_s) \right\},
    \qquad
    \kappa(\widetilde{A}\widetilde{P}_1) = \kappa(\widetilde{P}_1) = \sqrt{\kappa(\widetilde{A})}.
  \end{align}
\end{theorem}
\begin{proof}
  The condition numbers of $\widetilde{A}\widetilde{P}_{s}$ and $\widetilde{P}_{s}$ are
  \begin{align}\label{eq.tmp002}
    \kappa(\widetilde{A}\widetilde{P}_{s}) = \frac{c\lambda_{\max}\left( c\lambda_{\min} + s \right)}{c\lambda_{\min}\left( c\lambda_{\max} + s \right)},\qquad\kappa(\widetilde{P}_s) = \frac{c\lambda_{\max} + s}{c\lambda_{\min} + s}.
  \end{align}
  Since $\widetilde{A} \ne I$, $\kappa(\widetilde{A}\widetilde{P}_s)$ is a monotonically increasing function of $s$ while $\kappa(\widetilde{P}_s)$ is a monotonically decreasing function of $s$.
  Therefore, the solution $s$ of the equation
  \begin{align} \label{eq.tmp001}
    \kappa(\widetilde{A}\widetilde{P}_s) = \kappa(\widetilde{P}_s)
  \end{align}
  is the minimizer of $\max \{\kappa(\widetilde{A}\widetilde{P}_s), \kappa(\widetilde{P}_s)\}$.
  The equation \eqref{eq.tmp001} has one solution in $[0,\infty)$ because $\kappa(\widetilde{A}\widetilde{P}_0) = 1$, $\lim_{s \to \infty} \kappa(\widetilde{A}\widetilde{P}_s) = \kappa(\widetilde{A})$, $\kappa(\widetilde{P}_0) = \kappa(\widetilde{A})$, $\lim_{s \to \infty} \kappa(\widetilde{P}_s) = 1$.
  By substituting \eqref{eq.tmp002}, $c\lambda_{\max} = \sqrt{\kappa(A)}$, and $c\lambda_{\min} = 1/\sqrt{\kappa(A)}$ into \eqref{eq.tmp001}, we have a quadratic equation in $s$, and algebraic manipulations yield the solution $s = 1$.
  It is easily seen that $\kappa(\widetilde{P}_1) = \sqrt{\kappa(\widetilde{A})}$.
\end{proof}

After preconditioning, it is necessary to compute both $\log(\widetilde{A}\widetilde{P}_1)$ and $\log(\widetilde{P}_1)$. When the total number of abscissas is $m$, each will be computed with $m/2$ abscissas.
Here, since $\kappa(\widetilde{A}\widetilde{P}_1) = \kappa(\widetilde{P}_1) = \sqrt{\kappa(\widetilde{A})}$, the error of the GL quadrature for the two logarithms is $\mathcal{O}(\exp(-\rho(\sqrt{\kappa(\widetilde{A})})m/2))$, where $\rho(\cdot)$ is defined in \eqref{eq.err-gl-logm}.
Finally, the error of the $m$-point preconditioned GL quadrature is
\begin{align}
  \left\| {\log(\widetilde{A})
    - \sum_{k = 1}^{m/2}w_{k}F\left( t_{k};c'\widetilde{A}\widetilde{P}_1 \right)}
    + \sum_{k = 1}^{m/2} w_{k} F\left( t_{k};c''\widetilde{P}_1 \right)
    + \log(c)I
  \right\|_2
  \le K''\exp\left(-\rho\left(\sqrt{\kappa(\widetilde{A})}\right) \frac{m}{2}\right),
\end{align}
where $c', c''$ are defined at \eqref{eq:actual-preconditioning}, and $K''$ is a scalar constant independent of $m$.
Therefore, we can compare the convergence speed with and without the preconditioning by comparing  $\rho(\kappa(\widetilde{A}))$ to $\rho\left(\sqrt{\kappa(\widetilde{A})}\right)/2$.

Before the comparison, we also see the convergence on the DE formula for reference.
In a non-reviewed report \cite{tatsuoka2020convergence}, the error of the DE formula for $\log(\widetilde{A})$ is reported as 
\begin{align}
  \left\| {\log(\widetilde{A}) - \sum_{k = 1}^{m}w_{k}F\left( t_{k};\widetilde{A} \right)} \right\|_2 = K'''\exp(-\rho(\kappa(A)) m),
  \quad \rho(\kappa) = \frac{2\pi d_{0}(\sqrt{\kappa})}{r - l},
\end{align}
where, $t_k, w_k$ are the abscissas and the weights of the DE formula, and $K'''$ is a scalar constant independent of $m$,
\begin{align}
  d_{0}(\lambda) = \arcsin\left(\sqrt{\frac{\left( \log\lambda \right)^{2} + 2\pi^{2} - \sqrt{\left\lbrack \left( \log\lambda \right)^{2} + 2\pi^{2} \right\rbrack^{2} - 4\pi^{4}}}{2\pi^{2}}}\right),
\end{align}
and $l, r$ are determined by the error tolerance of the DE formula.
Because the convergence of the DE formula becomes slow as $\kappa(A)$ becomes large, we also have a choice to apply the preconditioning technique to the DE formula.

Figure \ref{fig.convergence_rate} illustrates the convergence speed of the GL quadrature (\texttt{GL}), the DE formula (\texttt{DE}), the preconditioned GL quadrature (\texttt{PGL}), and the preconditioned DE formula (\texttt{PDE}).
Here, the convergence speed means the constant $\rho$ in the error of the $m$-point quadrature formula $\mathcal{O}(\exp(-\rho m))$.
For \texttt{DE} and \texttt{PDE}, the error tolerance is set to $10^{-12}$.
In Figure \ref{fig.convergence_rate}, an algorithm with higher values on the $y$-axis converges faster.
For example, when $\kappa(\widetilde{A}) = 10^4$, the fastest algorithm may be \texttt{PGL} and the next one is \texttt{DE}.
The figure shows that \texttt{GL} will be the fastest when $\kappa(A) \lesssim 130$, \texttt{PGL} will be the fastest when $130 \lesssim \kappa(\widetilde{A}) \lesssim 3.0 \times 10^5$, and \texttt{DE} will be the fastest when $\kappa(\widetilde{A}) \gtrsim 3.0\times 10^5$.
The preconditioning is effective for the GL quadrature because its convergence speed is sensitive to changes in $\kappa(A)$.
For the opposite reason, the preconditioning will not be effective for the DE formula.

\begin{figure}
  \centering
  \includegraphics[width=10cm]{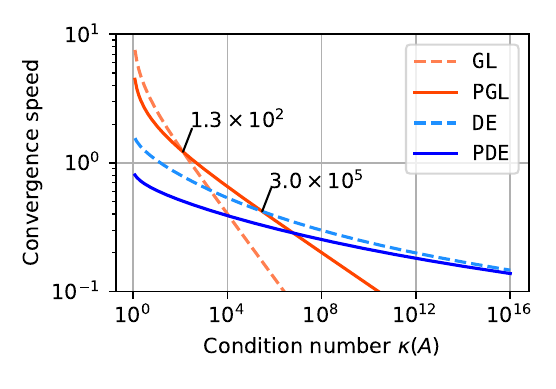}
  \caption{
    Convergence speed of quadrature formulas for $\log(A)$, i.e., the convergence of a quadrature formula is fast if the values in the graph are large.
  }
  \label{fig.convergence_rate}
\end{figure}

\section{Numerical experiments}
For numerical experiments, we used Julia 1.11.1 for the programs and ran them on a machine with a Core-i7 9700K CPU and 48GB RAM.
The source code and implementation details can be found on GitHub\footnote{\url{https://github.com/f-ttok/article-logm-preconditioning}}.

First, we check the convergence of \texttt{PGL}.
The test matrix is $A = \mathrm{tridiag}(-1, 2, -1)$, where $n = 200$ and $\kappa(A) \approx 1.6 \times 10^4$.
The reference solution is computed in arbitrary precision arithmetic with the \texttt{BigFloat} data type of Julia, which gives roughly 77 significant decimal digits.
The convergence profiles of \texttt{DE}, \texttt{GL}, and \texttt{PGL} are shown in Figure \ref{fig.history}.
Additionally, the error of the GL quadrature for $\log(\widetilde{A}\widetilde{P}_1)$ and $\log(\widetilde{P}_1)$ is also plotted in the figure.
Figure \ref{fig.history} demonstrates the fast convergence of \texttt{PGL} due to the efficient computation of $\log(\widetilde{A}\widetilde{P}_1)$ and $\log(\widetilde{P}_1)$.

Next, we computed the action of the matrix logarithm on a vector $\log(\widetilde{A}) \boldsymbol{b}$.
The test matrices are listed in Table \ref{tab:test_matrices}, and the vector $\boldsymbol{b}$ is $[1,\dots,1]^\top / \sqrt{n}$.
The number of abscissas was determined so that the error $\|\log(\widetilde{A})\boldsymbol{b} - \boldsymbol{x}\|_2$ is less than $10^{-12}$.
The linear systems in the integrand were solved using a sparse direct linear solver in \texttt{SparseArrays.jl}\footnote{\url{https://github.com/JuliaSparse/SparseArrays.jl}}.
The computational time and the number of evaluations of the integrand are listed in Table \ref{tab:cputime}.
We note that the time does not include the time to compute extreme eigenvalues of $A$.
Because the condition numbers of the test matrices range from $10^2$ to $10^6$, \texttt{PGL} is expected to be the fastest for most of the test matrices.
Indeed, the number of integrand evaluations of \texttt{PGL} was the smallest except for \texttt{gyro\_m} whose condition number is $1.2\times 10^6 > 3.0\times 10^5$, and the computation time of \texttt{PGL} was also the shortest except for \texttt{Pres\_Poisson}, \texttt{Dubcova1}, and \texttt{gyro\_m}.
The results show the effectiveness of the preconditioning.

\begin{figure}
  \centering
  \includegraphics[width=8cm]{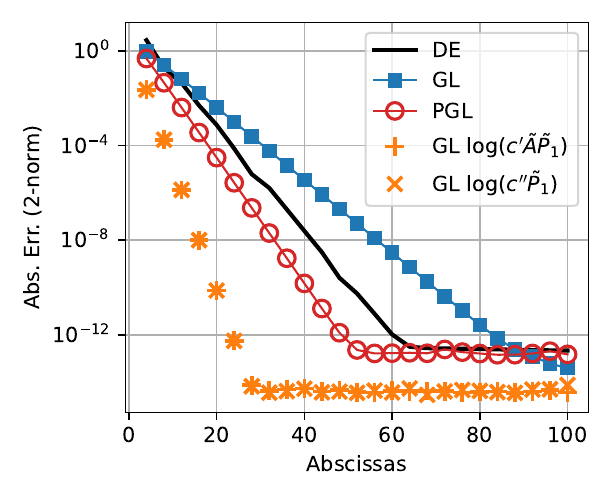}
  \caption{
    Convergence profiles of \texttt{DE}, \texttt{GL}, and \texttt{PGL}.
    We also plotted the error of the GL quadrature for $\log(c'\widetilde{A}\widetilde{P}_1)$ and $\log(c''\widetilde{P}_1)$ that are used in \texttt{PGL}.
  }
  \label{fig.history}
\end{figure}

\begin{table}[htbp]
  \centering
  \caption{Test matrices from SuiteSparse Matrix Collection \cite{davis2011university}}
  \begin{tabular}{lrl}
    \hline
    Matrix & Size $n$ & $\kappa(A)$\\
    \hline
    \texttt{Kuu} & 7102 & $3.4\times 10^4$\\
    \texttt{fv3} & 9801 & $2.0\times 10^3$\\
    \texttt{bundle1} & 10581 & $1.0 \times 10^3$\\
    \texttt{crystm02} & 13965 & $2.5\times 10^2$\\
    \texttt{Pres\_Poisson} & 14822 & $3.5\times 10^5$\\
    \texttt{Dubcova1} & 16129 & $6.8\times 10^4$\\
    \texttt{gyro\_m} & 17361 & $1.2\times 10^6$\\
    \texttt{bodyy5} & 18589 & $7.9\times 10^3$\\
    \texttt{bodyy6} & 19366 & $7.7\times 10^4$\\
    \hline
  \end{tabular}
  \label{tab:test_matrices}
\end{table}

\begin{table}[htbp]
  \centering
  \caption{
    Comparison of quadrature-based algorithms in terms of CPU time (in seconds) and the number of evaluations of the integrand (in parentheses).
    The results of the fastest algorithm among the four algorithms are written in bold face.
  }
  \begin{tabular}{l|rr|rr|rr|rr}
    \hline
    Matrix & \multicolumn{2}{|c|}{\texttt{GL}} & \multicolumn{2}{|c|}{\texttt{DE}} & \multicolumn{2}{|c|}{\texttt{PGL}} & \multicolumn{2}{|c}{\texttt{PDE}}\\
    \hline
    \texttt{Kuu} & 6.17 & (100) & 3.34 & (64) & \textbf{3.17} & \textbf{(54)} & 4.75 & (84)\\
    \texttt{fv3} & 0.57 & (49) & 0.65 & (53) & \textbf{0.43} & \textbf{(38)} & 0.86 & (72)\\
    \texttt{bundle1} & 1.89 & (41) & 2.21 & (48) & \textbf{1.55} & \textbf{(34)} & 3.16 & (70)\\
    \texttt{crystm02} & 1.80 & (29) & 2.75 & (45) & \textbf{1.63} & \textbf{(28)} & 4.12 & (68)\\
    \texttt{Pres\_Poisson} & 37.30 & (179) & \textbf{13.87} & \textbf{(76)} & 15.13 & (74) & 17.74 & (94)\\
    \texttt{Dubcova1} & 10.06 & (119) & \textbf{4.67} & \textbf{(64)} & 4.71 & (60) & 7.03 & (90)\\
    \texttt{gyro\_m} & 24.36 & (244) & \textbf{6.97} & \textbf{(81)} & 7.56 & (86) & 8.51 & (96)\\
    \texttt{bodyy5} & 1.99 & (69) & 1.85 & (59) & \textbf{1.29} & \textbf{(44)} & 2.24 & (74)\\
    \texttt{bodyy6} & 3.91 & (122) & 2.38 & (68) & \textbf{2.00} & \textbf{(60)} & 3.00 & (90)\\
    \hline
  \end{tabular}
  \label{tab:cputime}
\end{table}

\section{Conclusion}
We have presented a preconditioning technique of the GL quadrature for the logarithms of SPD matrices.
The condition numbers of the preconditioned matrices are the square root of the condition number of the original matrix.
In particular, when $130 \lesssim \kappa(A) \lesssim 3.0\times 10^5$, the preconditioned GL quadrature is faster than other quadrature-based algorithms.

\section*{Acknowledgments}
This work was supported by JSPS KAKENHI Grant Number 20H00581.


\end{document}